\documentclass[11pt, reqno]{amsart}

\usepackage{enumerate}
\usepackage{amsmath}%
\usepackage{amsfonts}%
\usepackage{amssymb}%
\usepackage{graphicx}
\usepackage{mathrsfs}
\usepackage{hyperref}
\usepackage[margin=1in]{geometry}
\usepackage{subfigure}
\usepackage{lineno} 

\usepackage{pgfplots} 
\usepackage{tikz}
\usetikzlibrary{calc}
\usetikzlibrary{arrows,shapes,chains}
\pgfdeclarelayer{background}
\pgfdeclarelayer{foreground}
\pgfsetlayers{background,main,foreground}
\usetikzlibrary{decorations.pathreplacing}

\newtheorem{theorem}{Theorem}
\theoremstyle{plain}

\newtheorem{claim}[theorem]{Claim}

\newtheorem{construction}[theorem]{Construction}

\newtheorem{conjecture}[theorem]{Conjecture}
\newtheorem{corollary}[theorem]{Corollary}

\newtheorem{lemma}[theorem]{Lemma}

\newtheorem{proposition}[theorem]{Proposition}

\numberwithin{equation}{section}
\numberwithin{theorem}{section}
\numberwithin{case}{section}

\newtheorem*{lemma*}{Lemma~\ref{lem:abs1}}
\newtheorem*{Lemma*}{Lemma~\ref{lem:abs2}}

\def \cP{\mathcal{P}}

\def \bfi{\mathbf{i}}
\def \bfu{\mathbf{u}}
\def \bfv{\mathbf{v}}

\def\COMMENT#1{}
\let\COMMENT=\footnote
\begin{document}

\title{Matching of given sizes in hypergraphs}

\author{Yulin Chang, Huifen Ge, Jie Han$^{\S}$, and Guanghui Wang}

\thanks{The research of the first author was supported in part by National Natural Science Foundation of China (12071260).
The research of the second author was supported in part by Science Found of Qinghai Province (2021-ZJ-703).
The research of the fourth author was supported in part by Natural Science Foundation of China (11631014) and Shandong University Multidisciplinary Research and Innovation Team of Young Scholars.}
\thanks{$^{\S}$Corresponding author.}
\thanks{\emph{E-mail address:} ylchang93@163.com (Y. Chang), gehuifen11856@163.com (H. Ge), hanjie@bit.edu.cn (J. Han), ghwang@sdu.edu.cn (G. Wang)}

\address{Y. Chang, Data Science Institute, Shandong University, Jinan, 250100, China}
\address{H. Ge, School of Computer, Qinghai Normal University, Xining, 810001, China}
\address{J. Han, School of Mathematics and Statistics, Beijing Institute of Technology, Beijing, 100000, China}
\address{G. Wang, School of Mathematics and Data Science Institute, Shandong University, Jinan, 250100, China}

\begin{abstract}
For all integers $k,d$ such that $k \geq 3$ and $k/2\leq d \leq k-1$, let $n$ be a sufficiently large integer {\rm(}which may not be divisible by $k${\rm)} and let $s\le \lfloor n/k\rfloor-1$.
We show that if $H$ is a $k$-uniform hypergraph on $n$ vertices with $\delta_{d}(H)>\binom{n-d}{k-d}-\binom{n-d-s+1}{k-d}$, then $H$ contains a matching of size $s$.
This improves a recent result of Lu, Yu, and Yuan and also answers a question of K\"uhn, Osthus, and Townsend.
In many cases, our result can be strengthened to $s\leq \lfloor n/k\rfloor$, which then covers the entire possible range of $s$.
On the other hand, there are examples showing that the result does not hold for certain $n, k, d$ and $s= \lfloor n/k\rfloor$.
\end{abstract}

\date{\today}

\maketitle

\section{Introduction}
Given $k\ge 2$, a \emph{$k$-uniform hypergraph} (for short, \emph{$k$-graph}) $H$ consists of a vertex set $V(H)$ and an edge set $E(H)\subseteq \binom{V(H)}{k}$, where every edge is a $k$-subset of $V(H)$.
A \emph{matching} (or \emph{integer matching}) in $H$ is a collection of vertex-disjoint edges of $H$.
A \emph{perfect matching} in $H$ is a matching that covers all vertices of $H$.
Let $|V(H)|=n$. Clearly, a perfect matching in $H$ exists only if $k$ divides $n$.
When $k$ does not divide $n$, we call a matching $M$ in $H$ a \emph{near perfect matching} if $|M|=\lfloor n/k \rfloor$.

Let $H$ be a $k$-graph.
For a $d$-subset $S$ of $V(H)$, where $1 \leq d\leq k-1$, we define $\deg_H (S)$ to be the number of edges in $H$ containing $S$.
The \emph{minimum $d$-degree $\delta _{d}(H)$} of $H$ is the minimum of $\deg_H (S)$ over all $d$-subsets $S$ of $V(H)$.
We refer to $\delta _{1}(H)$ as the \emph{minimum vertex degree} of $H$ and $\delta _{k-1}(H)$ as the \emph{minimum codegree} of $H$.

The study of perfect matchings is one of the fundamental problems in combinatorics.
In the case of graphs, that is, $k=2$, a theorem of Tutte~\cite{Tu47} gives necessary and sufficient conditions for $H$ to contain a perfect matching, and Edmonds' Algorithm \cite{Edmonds} finds such a matching in polynomial time.
However, for the case $k\ge 3$, the decision problem whether a $k$-graph contains a perfect matching is famously NP-complete (see~\cite{garey,karp}).

\subsection{Perfect matchings}
The following conjecture from~\cite{HPS, KuOs-survey} gives a minimum $d$-degree condition that ensures a perfect matching in a $k$-graph.

\begin{conjecture}\label{conj1}
Let $k,d\in \mathbb N$ and $1\le d\le k-1$.
Then there is an $n_0 \in \mathbb N$ such that the following holds for all $n\ge n_0$.
Suppose $H$ is a $k$-graph on $n\in k\mathbb N$ vertices with
\[
\delta _d (H) \geq \left ( \max \left \{ \frac{1}{2} , 1-\left (1-\frac{1}{k}\right)^{k-d} \right\}+o(1)\right) \binom{n-d}{k-d},
\]
then $H$ contains a perfect matching.
\end{conjecture}

There are two types of extremal examples, namely, the so-called \emph{divisibility barrier} and \emph{space barrier} which show, if true, the minimum degree conditions in Conjecture~\ref{conj1} are asymptotically best possible.

\begin{construction}[\cite{TrZh12}, Divisibility Barrier]\label{const1}
Fix integers $j,k,n$ such that $j \in \{0,1\}$ and $n\in k\mathbb N$.
Let $V$ be a set of size $n$ with a partition $U\cup W$ such that $|U| \not\equiv jn/k \bmod 2$.
Let $H^j$ be the $k$-graph on $V$ whose edges are $k$-sets $e$ such that $|e \cap U|\equiv j \bmod 2$.
\end{construction}

To see why there is no perfect matching in Construction~\ref{const1}, it is not hard to see that $\delta _{d} (H^j)\le(1/2+o(1)) \binom{n-d}{k-d}$ and the equality is attained when $|U|\approx |V| \approx n/2$.
For the case $j=0$, note that any matching in $H^0$ covers an even number of vertices in $U$.
Then, due to the parity of $|U|$, $H^0$ does not contain a perfect matching.
For the case $j=1$, suppose that there exists a perfect matching $M$ in $H^1$.
Note that $M$ has $n/k$ edges, and each edge $e$ in $M$ satisfies $|e\cap U|\equiv 1 \bmod 2$.
Summing over all edges in $M$, we obtain that $|U| \equiv n/k \bmod 2$, contradicting to our assumption on $|U|$.
So there cannot exist a perfect matching in $H^1$.
Moreover, it is known that one can construct such divisibility barriers with any finite number of parts (but with smaller minimum $d$-degrees).

\begin{construction}[Space Barrier]\label{const2}
Fix integers $k,n$ such that $n\in k\mathbb N$.
Let $V$ be a vertex set of size $n$ with a partition $U \cup W$, and $|W|=s<n/k$.
Let $H_k^k(U, W)$ be the $k$-graph on $V$ whose edges are all $k$-sets that intersect $W$.
\end{construction}
Note that any matching in Construction~\ref{const2} has at most $s<n/k$ edges, so there cannot exist a perfect matching in $H_k^k(U, W)$.
For $1\le d \le k-1$, it is easy to see that $\delta _d (H_k^k(U, W))=\binom{n-d}{k-d}-\binom{n-d-s}{k-d}=\left(1-\left(1-s/n\right)^{k-d} +o(1)\right)\binom{n-d}{k-d}$.
Moreover, the maximum value of $\delta _d (H_k^k(U, W))$ is attained by $s =\lceil n/k\rceil- 1$, which gives the second term in Conjecture~\ref{conj1}.

Conjecture~\ref{conj1} has attracted a great deal of attention in recent years, see results, e.g.,~\cite{KO06mat, Pik, RRS06mat}.
It has been confirmed for all $0.375k \leq d \leq k-1$ and for $1 \leq k-d \leq 4$ and for $(k,d)\in \{(12,5),(17,7)\}$.
In particular, R\"odl, Ruci\'nski, and Szemer\'edi~\cite{RRS09} determined that the minimum codegree threshold is $n/2-k+C$ for sufficiently large $n \in k \mathbb{N}$, where $C \in \{3/2,2,5/2,3\}$ depends on the values of $n$ and $k$.
Furthermore, the exact thresholds for sufficiently large $n$ are obtained for these cases except $(k,d)\in \{(5,1),(6,2)\}$ in \cite{CzKa, FK, JieNote, Khan1, Khan2, KOT, MaRu, TrZh12, TrZh13, TrZh15}.
For more results we refer the reader to the excellent surveys \cite{RR, zsurvey}.

\subsection{Matchings of other sizes}
When $k\nmid n$, R\"odl, Ruci\'nski, and Szemer\'edi~\cite{RRS09} showed that a minimum codegree roughly $n/k$ guarantees the existence of a matching of size $\lfloor n/k\rfloor$.
This stands as a steep contrast to the threshold for perfect matchings, which is roughly $n/2$.
Since then the study of near perfect matchings and matchings of general size $s<n/k$ has also attracted a lot of attention.
Indeed, R\"odl, Ruci\'nski, and Szemer\'edi~\cite{RRS09} determined the codegree threshold for the existence of a matching of size $s$ in a $k$-graph $H$ for all $s \leq \lfloor n/k \rfloor-k+2$.
Han~\cite{JieNear} extended this result to all $s<n/k$, verifying a conjecture of R\"odl, Ruci\'nski, and Szemer\'edi.
Moreover, Han \cite{Han15_mat} gave a divisibility barrier construction that prevents the existence of near perfect matchings in $H$ which generalizes Construction~\ref{const1}.
He also proposed a conjecture on the minimum $d$-degree threshold forcing a (near) perfect matching in $H$ which generalizes Conjecture~\ref{conj1}.
In the same paper, He determined the minimum $(k-2)$-degree threshold and gave an upper bound and lower bound for general $d$-degree threshold.
K\"uhn, Osthus, and Treglown \cite{KOT} determined the vertex degree threshold for the existence of a matching of size $s$ in a $3$-graph $H$ for all $s \leq n/3$, which can be seen as a strengthening of an old result of Bollob{\'a}s, Daykin, and Erd\H{o}s \cite{BDE76} for 3-graphs.

K\"uhn, Osthus, and Townsend \cite{KOTo} determined the $d$-degree threshold for the existence of a matching of size $s$ in a $k$-graph $H$ asymptotically for $1 \leq d \leq k-2$ and $s \leq \min\{n/(2k-2d), (1-o(1))n/k\}$ (when $k/2<d\le k-2$, this is equivalent to $s\le (1-o(1))n/k$).
They asked whether the $(1-o(1))n/k$ can be replaced by $n/k-C$ for some constant $C$ depending only on $d$ and $k$.
A recent result of Lu, Yu, and Yuan \cite{Lu} shows that one can take
\[
C=(1-d/k)\left\lceil\frac{k-d}{2d-k}\right\rceil,
\]
which answers the aforementioned question.
However, the term $(1-d/k)\lceil\frac{k-d}{2d-k}\rceil$ can be arbitrarily large when $d$ is close to $k/2$.
In this paper we reduce this gap by showing that one can take $C=1$ (see Corollary~\ref{coro}), and in many cases, one can actually take $C=0$.

In fact, Lu, Yu, and Yuan \cite{Lu} determined the \emph{exact} minimum $d$-degree threshold for the existence of a matching of size $s$ in a $k$-graph $H$ for $k/2<d\le k-1$ and $n/k-o(n)\le s\le n/k-(1-d/k)\lceil\frac{k-d}{2d-k}\rceil$.
Our main improvements can be summarized in the following two results.
The first one says that the largest $s$ we can take is either $\lfloor n/k\rfloor$ (clearly the best possible) or $\lfloor n/k\rfloor-1$ (best possible for certain values of $k$, $d$ and $n$).

\begin{theorem}\label{main}
For all integers $k,d$ such that $k \geq 3$ and $k/2\leq d \leq k-1$, there exists an $n_0 \in \mathbb N$ such that the following holds for $n\ge n_0$.
Let $n\equiv r \bmod k$ be an integer, where $0\le r\le k-1$.
Suppose $H$ is an $n$-vertex $k$-graph with $\delta_{d}(H) > \binom{n-d}{k-d}-\binom{n-d-s+1}{k-d}$, where
\begin{displaymath}
s = \left\{ \begin{array}{ll}
\lfloor n/k\rfloor, & \textrm{if $k/2\le d<\lceil 2k/3\rceil$ and $r\ge 2$, or if $\lceil 2k/3\rceil\le d\le k-1$ and} \\ &\textrm{$r\ge k-d$};\\
\lfloor n/k\rfloor-1, & \textrm{otherwise}.
\end{array} \right.
\end{displaymath}
Then $H$ contains a matching of size $s$.
In particular, $H$ contains a matching covering all but at most $2k-d-1$ vertices.
\end{theorem}

\noindent\textbf{Remark.} The theorem also holds for $s=\lfloor n/k\rfloor$ in the case $k/2\le d\le 0.59k$ and $r=1$ by a result of Han~\cite{Han15_mat}, which we did not include above.
On the other hand, one can not hope for $s=\lfloor n/k\rfloor$ in the case $0.59k < d \le k-1$ and $r=1$ due to a construction~\cite[Proposition 1.11]{Han15_mat}.
However, for other cases it is not clear whether the theorem holds for $s=\lfloor n/k\rfloor$.

The second result is on matchings of any given size $s$ in $H$ for all $s\le \lfloor n/k\rfloor-1$, which can be easily deduced from Theorem~\ref{main}.
In particular, it says that one can take $C=1$ in the aforementioned question of K\"uhn, Osthus, and Townsend \cite{KOTo}.

\begin{corollary}\label{coro}
For all integers $k,d$ such that $k \geq 3$ and $k/2\leq d \leq k-1$, there exists an $n_0 \in \mathbb N$ such that the following holds for $n\ge n_0$.
Let $n$ be an integer {\rm(}which may not be divisible
by $k${\rm)} and let $s$ be an integer such that $s\le \lfloor n/k\rfloor-1$.
Suppose $H$ is an $n$-vertex $k$-graph with $\delta_{d}(H) > \binom{n-d}{k-d}-\binom{n-d-s+1}{k-d}$.
Then $H$ contains a matching of size $s$.
\end{corollary}

\begin{proof}
For $s\le \lfloor n/k\rfloor-1$, let $t$ be an integer such that $t=\lfloor (n+t)/k\rfloor-s-1$.\footnote{Note that one can choose such a $t$ by trying $t=0,1,2,\dots$}
Then $s+t=\lfloor (n+t)/k\rfloor-1$.
Consider the auxiliary $k$-graph $H'$ by adding $t$ vertices to $H$ such that $H'$ contains all edges of $H$ and all $k$-sets containing any of these new vertices.
Note that $H'$ has $n+t$ vertices and
\begin{align}
\delta_d(H')&=\delta_d(H)+\binom{n+t-d}{k-d}-\binom{n-d}{k-d} \nonumber \\
 &>\binom{n-d}{k-d}-\binom{n-d-s+1}{k-d}+\binom{n+t-d}{k-d}-\binom{n-d}{k-d}\nonumber \\
 &=\binom{(n+t)-d}{k-d}-\binom{(n+t)-d-(s+t-1)}{k-d}.\nonumber
\end{align}
Thus, we apply Theorem~\ref{main} on $H'$ and conclude that $H'$ contains a matching $M'$ of size $s+t=\lfloor (n+t)/k\rfloor-1$.
By deleting at most $t$ edges from $M'$ that contain the new vertices, we can get a matching in $H$ of size $s$ and we are done.
\end{proof}

Combining Theorem~\ref{main} and Corollary~\ref{coro}, we obtain the exact minimum $d$-degree threshold for a matching of size $s$ for all $s\le n/k$ if $k/2\le d\le 0.59k$ and $r=1$, or if $k/2\le d<\lceil 2k/3\rceil$ and $2\le r\le k-1$, or if $\lceil 2k/3\rceil\le d\le k-1$ and $k-d\le r\le k-1$.

At last, we remark that if we only target on the~\emph{asymptotical} minimum $d$-degree thresholds the problem would be much easier.
To formulate such a result, let us introduce the notion of fractional matchings.
Given a $k$-graph $H=(V, E)$, a \emph{fractional matching} in $H$ is a function $\omega: E \to [0,1]$ such that for each $v\in V(H)$ we have that $\sum_{e \ni v} w(e)\leq1$.
Then $\sum _{e \in E(H)} w(e)$ is the \emph{size} of $w$.
If the size of $w$ in $H$ is $n/k$ then we say that $w$ is a \emph{perfect fractional matching}.
Given $k,d \in \mathbb N$ such that $d \leq k-1$, define $c^*_{k,d}$ to be the smallest number $c$ such that every $k$-graph $H$ on $n$ vertices with $\delta _{d} (H) \geq (c +o(1)) \binom{n-d}{k-d}$ contains a perfect fractional matching.
Alon et al.~\cite{AFHRRS} conjectured that $c^*_{k,d}=1-(1-1/k)^{k-d}$ for all $d, k \in \mathbb N$ and verified the case $k-d\le 4$.
So far the conjecture was verified when $d\ge 0.4k$ and when $k-d\le 4$~\cite{KOTo, FK}.

\begin{theorem}\label{thm:asym}
Let $k, d$ be integers such that $1\le d\le k-1$ and $\gamma >0$, then there exists an $n_0\in\mathbb{N}$ such that the following holds for $n\ge n_0$.
Suppose $H$ is an $n$-vertex $k$-graph with $\delta_d(H)\ge (c_{k,d}^*+\gamma)\binom{n-d}{k-d}$, then $H$ contains a matching $M$ that covers all but at most $2k-d-1$ vertices.
In particular, when $n\in k\mathbb N$, $M$ is a perfect matching or covers all but exactly $k$ vertices.
\end{theorem}

Theorem~\ref{thm:asym} has been proved by Han and Treglown~\cite[Theorem 7.2]{HT} under the assumption $\delta_d(H)\ge (\max\{c_{k,d}^*, 1/3\}+\gamma)\binom{n-d}{k-d}$.
Using the regularity method, under the assumption of Theorem~\ref{thm:asym}, it is easy to construct a matching that leaves $o(n)$ vertices uncovered.
Here we apply the absorbing method and thus reduce the number of uncovered vertices to at most $2k-d-1$.

As a typical approach in this area, we employ the absorbing method, popularized by R\"{o}dl, Ruci\'{n}ski, and Szemer\'{e}di~\cite{RRS06}, and distinguish the ``extremal'' and ``non-extremal'' cases, which is also the approach used by Lu, Yu, and Yuan~\cite{Lu}.
For example, the extremal case has been resolved by Lu, Yu, and Yuan~\cite{Lu}.
The main difference is in the part of absorption in the non-extremal case, where we deal with large and small values of $d$ with different strategies: we use the lattice-based absorbing method of Han~\cite{Han15_mat} when $k/2\le d< \lceil 2k/3\rceil$, and use an argument of R\"{o}dl, Ruci\'{n}ski, and Szemer\'{e}di~\cite{RRS06} when $\lceil 2k/3 \rceil \leq d \leq k-1$.
The first method was initially used in \cite{Han14_poly} for solving a complexity problem of Karpi\'{n}ski, Ruci\'{n}ski, and Szyma\'{n}ska~\cite{KRS10}.

\medskip
\noindent\textbf{Organization.} Throughout the rest of the paper, $k$ denotes an integer with $k\ge 3$.
As usual, for an integer $b$, let $[b] = \{1,\dots,b\}$.
The rest of the paper is organized as follows.
In Section~\ref{sec2}, we first give some useful lemmas and then prove Theorem~\ref{main} and Theorem~\ref{thm:asym}.
In Sections~\ref{sec3} and~\ref{sec4}, we give the proofs of our absorbing lemmas (Lemma~\ref{lem:abs1} and Lemma~\ref{lem:abs2}).

\section{Proofs of Theorems~\ref{main} and \ref{thm:asym}}\label{sec2}
We first extend the definition of $H_k^k(U,W)$ in Construction~\ref{const2} to $H_k^{\ell}(U,W)$ for all $\ell \in [k]$.
Again, let $V$ be a vertex set with a partition $U \cup W$. Let $H_{k}^{\ell}(U, W)$ be the $k$-graph on $V$ whose edges are all $k$-sets $e$ such that $1\le |e\cap W|\le \ell$ (see Figure~\ref{fig}).
Given two $k$-graphs $H_1, H_2$ and a real number $\varepsilon >0$, we say that $H_2$ is \emph{$\varepsilon$-close} to $H_1$ if $V(H_1)=V(H_2)$ and $|E(H_1) \setminus E(H_2)| \leq \varepsilon |V(H_1)|^k$.

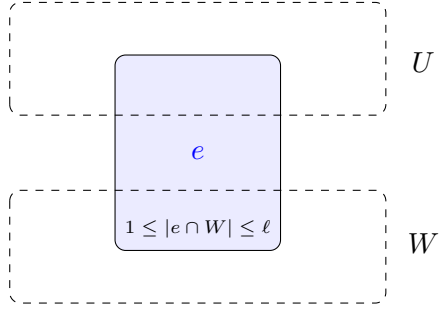
\begin{figure}[!ht]
\begin{center}
\begin{tikzpicture}
\begin{pgfonlayer}{background}    
\draw[rounded corners,,fill=blue!8] (-1.1,-0.8) rectangle (1.1, 1.8);  
\draw[rounded corners,dashed] (-2.5,1) rectangle (2.5, 2.5);  
\draw[rounded corners,dashed] (-2.5,0) rectangle (2.5, -1.5);  
\end{pgfonlayer}
\node at (3, 1.7) {$U$};
\node at (3, -0.7) {$W$};
\node at (0, 0.5) [color=blue]{$e$};
\node at (0, -0.5) {\tiny{$1\le |e\cap W|\le \ell$}};
\end{tikzpicture}
\caption{\small An illustration of the $k$-graph $H_{k}^{\ell}(U, W)$.}
\label{fig}
\end{center}
\end{figure}
We write $x\ll y\ll z$ to mean that we can choose constants from right to left, that is, there exist functions $f$ and $g$ such that, for any $z>0$, whenever $y\leq f(z)$ and $x\leq g(y)$, the subsequent statement holds.
Statements with more variables are defined similarly.
The following result by Lu, Yu, and Yuan~\cite{Lu} is needed for the ``extremal'' case.

\begin{lemma}[{\cite[Lemma 2.3]{Lu}}]\label{Lulemma2.3}
Given integers $d$ and $k$ such that $d\in [k-1]$, suppose $0<1/n\ll\varepsilon\ll 1/k$ and let $s\in\{\lfloor n/k\rfloor-1,\lfloor n/k\rfloor\}$.
Suppose $H$ is an $n$-vertex $k$-graph with $\delta_{d}(H) >\binom{n-d}{k-d}-\binom{n-d-s+1}{k-d}$, and $H$ is $\varepsilon$-close to
$H_{k}^{k-d}(U, W)$, where $|W|=s-1$ and $|U|=n-s+1$.
Then $H$ contains a matching of size $s$.
\end{lemma}

We also need the following lemma in the ``non-extremal'' case, which guarantees that the resulting $k$-graph has an almost perfect matching after taking away an absorbing matching (from Lemma~\ref{lem:abs1} and Lemma~\ref{lem:abs2}).

\begin{lemma}[{\cite[Lemma 5.7]{Lu}}]\label{lem:almost}
Given $\varepsilon, \sigma>0$ and integers $k,d$ such that $k/2\le d\le k-1$, suppose $0<1/n\ll\rho \ll\varepsilon,1/k$ and let $s\in\{\lfloor n/k\rfloor-1,\lfloor n/k\rfloor\}$.
Suppose $H$ is an $n$-vertex $k$-graph with $\delta_{d}(H) \ge \binom{n-d}{k-d}-\binom{n-d-s+1}{k-d}-\rho n^{k-d}$, and $H$ is not $\varepsilon$-close to $H_{k}^{k-d}(U, W)$ for any partition $U, W$ of $V(H)$ with $|W|=s-1$.
Then $H$ contains a matching covering all but at most $\sigma n$ vertices.
\end{lemma}

Now we state our absorbing lemmas, which are the new ingredients in the proof of Theorem~\ref{main}.
We postpone their proofs to Sections~\ref{sec3} and~\ref{sec4}.
Indeed, a known method of R\"{o}dl, Ruci\'{n}ski, and Szemer\'{e}di~\cite{RRS06} can be used to prove the absorbing lemma for the case $\lceil 2k/3 \rceil \leq d \leq k-1$ (Lemma~\ref{lem:abs2}), but does not apply for smaller $d$.
Fortunately, for smaller values of $d$ the quantity $\delta_d(H)$ is larger, in particular, $\delta_d(H)\ge (1/4+o(1))\binom{n-d}{k-d}$.
This allows us to use the lattice-based absorption method, developed by Han~\cite{Han15_mat}, and prove the absorbing lemma for the case $\min\{3,k/2\}\le d< \lceil 2k/3\rceil$ (Lemma~\ref{lem:abs1}).

\begin{lemma}\label{lem:abs1}
For $\min\{3,k/2\}\le d< \lceil 2k/3\rceil$ and $\gamma>0$, suppose $0<1/n\ll\alpha\ll \gamma,1/k$.
Let $H$ be an $n$-vertex $k$-graph with $\delta_d(H)\ge (1/4+\gamma)\binom{n-d}{k-d}$.
Then there exists a matching $M$ in $H$ of size $|M|\le \gamma n/k$ such that for any subset $R\subseteq V(H)\setminus V(M)$ with $|R|\le \alpha^{2} n$, $H[R\cup V(M)]$ contains a matching covering all but at most $k+1$ vertices.
\end{lemma}

\begin{lemma}\label{lem:abs2}
For $\lceil 2k/3 \rceil \leq d \leq k-1$ and $\gamma'> 0$, suppose $0<1/n\ll\beta\ll \gamma',1/k$.
Let $H$ be an $n$-vertex $k$-graph with $\delta _d (H) \geq \gamma' n^{k-d}$, then there exists a matching $M'$ in $H$ of size $|M'| \leq \beta n/k$ and such that for any subset $R\subseteq V(H)\setminus V(M')$ with $|R|\leq\beta^2 n$, $H[V(M')\cup R]$ contains a matching covering all but at most $2k-d-1$ vertices.
\end{lemma}

Now we combine Lemmas~\ref{Lulemma2.3},~\ref{lem:almost},~\ref{lem:abs1}, ~\ref{lem:abs2} and prove Theorem~\ref{main}.

\begin{proof}[Proof of Theorem~\ref{main}]
Given $k\geq 3$ and $k/2\leq d \leq k-1$.
Let $\varepsilon>0$ be as in Lemma~\ref{Lulemma2.3}, and we choose additional constants satisfying the following hierarchy depending on the value of $d$:
$$0<\sigma\ll \eta=\gamma\ll\varepsilon\ll 1/k \mbox{ if } k/2\le d< \lceil 2k/3\rceil;$$
and
$$0<\sigma\ll \eta\ll \gamma',\varepsilon\ll 1/k \mbox{ if } \lceil 2k/3 \rceil \leq d \leq k-1.$$
More precisely, take $\eta:=\gamma$ and $\sigma:=\alpha^2\ll \gamma$ in the case $\min\{3,k/2\}\le d<\lceil2k/3\rceil$, where $\alpha$ is given by Lemma~\ref{lem:abs1}; and take $\eta:=\beta\ll \gamma'$ and $\sigma:=\beta^2\ll \eta$ in the case $\lceil2k/3\rceil\le d\le k-2$, where $\beta$ is given by Lemma~\ref{lem:abs2}.
We also assume that $n$ is sufficiently large.
Let $H=(V,E)$ be an $n$-vertex $k$-graph satisfying the degree condition as in Theorem~\ref{main}.
We may assume $d\le k-2$ since the case $d=k-1$ has been proved by Han~\cite[Theorem 1.1]{Han15_mat}.
The extremal case immediately follows from Lemma~\ref{Lulemma2.3}.
Note that this is the only place where the exact $d$-degree condition is needed.
Thus, we may assume that $H$ is not $\varepsilon$-close to $H_k^{k-d}(U,W)$ for any partition $U,W$ of $V$ with $|W| = s-1$.

Note that when $\min\{3,k/2\}\le d<\lceil2k/3\rceil$, we have
\[
\binom{n-d-s+1}{k-d}\le \left( \frac{n-s+1}{n-k+1} \right)^{k-d} \binom{n-d}{k-d} \le \left( 1-\frac{s-k}{n-k+1} \right)^{k/3} \binom{n-d}{k-d}
\]
Using $1-x\le e^{-x}$ and $sk\ge n-2k$ (as $s\ge n/k-2$), we infer
\[
\binom{n-d-s+1}{k-d} \le e^{-\frac{s-k}{n-k+1}\cdot \frac{k}{3}}\binom{n-d}{k-d} \le e^{-\frac{n-2k-k^2}{n-k+1}\cdot \frac{1}{3}}\binom{n-d}{k-d}\le e^{-0.33}\binom{n-d}{k-d},
\]
where we used that $\frac{n-2k-k^2}{n-k+1}\ge 0.99$ as $n$ is large.
We conclude that $\binom{n-d}{k-d}-\binom{n-d-s+1}{k-d}\ge (1/4+\gamma)\binom{n-d}{k-d}$ as $e^{-0.33}\approx 0.71$ and $\gamma\ll 1/k$.
On the other hand, when $\lceil2k/3\rceil\le d\le k-2$, we have $\binom{n-d}{k-d}-\binom{n-d-s+1}{k-d}\ge \gamma' n^{k-d}$ by the choice of $\gamma'$.

First we find an absorbing set in $H$.
Note that $2k-d-1\ge k+1$ since $d\le k-2$.
By Lemma~\ref{lem:abs1} and Lemma~\ref{lem:abs2}, there exists a matching $M$ in $H$ of size $|M|\le \eta n/k$ such that, for every subset $R\subseteq V\setminus V(M)$ with $|R|\le \sigma n$, the induced $k$-graph $H[R\cup V(M)]$ contains a matching covering all but at most $\max \{k+1, 2k-d-1\}=2k-d-1$ vertices.

Let $H_1 = H[V\setminus V(M)]$ and $n_1=n-k|M|\ge (1-\eta) n$. Next we find an almost perfect matching in $H_1$.
Recall that $\delta_d(H)> \binom{n-d}{k-d}-\binom{n-d-s+1}{k-d}$, then we have that
\[
\delta_d(H_1)\ge \delta_d(H)-(\eta n)n^{k-d-1}\ge \binom{n_1-d}{k-d}-\binom{n_1-d-s+1}{k-d}-2\eta n_1^{k-d}
\]
and $H_1$ is not $(\varepsilon/2)$-close to $H_k^{k-d}(U,W)$ for any partition $U,W$ of $V(H_1)$ with $|W| = s-1$, since $n$ is sufficiently large and $\eta\ll \varepsilon$.
We now apply Lemma~\ref{lem:almost} with input $\sigma$ and $\varepsilon/2$ in place of $\varepsilon$, $H_1$ has a matching $M_1$ covering all but at most $\sigma n$ vertices of $V(H_1)$.
Denote by $U$ the set of these remaining vertices of $V(H_1)$. Then $|U|\le \sigma n$.

Let $\ell k+r=n$, where $0\le r\le k-1$.
Recall that $M$ is an absorbing matching in $H$, which implies that $H[V(M)\cup U]$ contains a matching $M_2$ covering all but at most $2k-d-1$ vertices.
Let $m$ be the number of unmatched vertices by $M_1\cup M_2$.
Then for the case $k/2\le d< \lceil 2k/3\rceil$ we have $m\le k+1$ by Lemma~\ref{lem:abs1};
more precisely, in this case $m=k+r$ when $r\in\{0,1\}$ and $m=r$ when $2\le r\le k-1$. Moreover, for the case $\lceil 2k/3\rceil\le d\le k-1$ we have $m\le 2k-d-1$ by Lemma~\ref{lem:abs2}; more precisely, in this case $m=k+r$ when $0\le r< k-d$ and $m=r$ when $k-d\le r\le k-1$.
These imply that $M_1\cup M_2$ is the desired matching.
\end{proof}

The following result is a weaker version of Lemma 5.6 in ~\cite{KOTo}.
It allows us to convert the fractional matchings into integer ones, and was proved by the weak hypergraph regularity lemma (and alternative proofs avoiding the regularity method are also known, see e.g. \cite{AFHRRS}).

\begin{lemma}[\cite{KOTo}]\label{lemma5.6}
 Let $k \geq 2$ and $1\leq d \leq k-1$ be integers, and let $\varepsilon>0$.
 Suppose that for some $b,c\in (0,1)$ and some $n_0\in \mathbb{N}$, every $k$-graph on $n \geq n_0$ vertices with $\delta_d(H)\geq c\binom{n-d}{k-d}$ has a fractional matching of size $(b+\varepsilon)n$.
 Then there exists an $n_0'\in \mathbb{N}$ such that every $k$-graph $H$ on $n \geq n_0'$ vertices with $\delta_d(H)\geq (c+\varepsilon)\binom{n-d}{k-d}$ has an integer matching of size at least $bn$.
\end{lemma}

As a consequence we obtain the following result by the definition of $c_{k,d}^*$.

\begin{lemma}\label{lem:fractional}
Let $k, d$ be integers such that $1\le d\le k-1$ and $\gamma,\sigma>0$, the following holds for sufficiently large $n$.
Suppose $H$ is an $n$-vertex $k$-graph with $\delta_d(H)\ge (c_{k,d}^*+\gamma)\binom{n-d}{k-d}$, then $H$ contains a matching $M$ that covers all but at most $\sigma n$ vertices.
\end{lemma}

\begin{proof}
Let $k,d$ and $\gamma$ be given, and let $n_0$ and $n_0'$ be given by Lemma~\ref{lemma5.6}.
Without loss of generality we may assume that $1/n\ll\sigma \ll \gamma, 1/k$.
Note that every $k$-graph $G$ with $\delta_d(G)\geq (c_{k,d}^*+\sigma/k)\binom{n-d}{k-d}$ has a perfect fractional matching of size $n/k$ by the definition of $c_{k,d}^*$; and then we use Lemma~\ref{lemma5.6} to transform this into an almost perfect integer matching.
More precisely, since $\delta_d(H)\ge (c_{k,d}^*+\sigma)\binom{n-d}{k-d}$ by the choice of $\sigma$, applying Lemma~\ref{lemma5.6} on $H$ with $b=(1-\sigma)/k$, $c=c^*_{k,d}+\sigma/k$ and $\varepsilon=\sigma/k$, we conclude that $H$ contains an integer matching of size at least $(1-\sigma)n/k$, that is, a matching covering all but at most $\sigma n$ vertices, as required.
\end{proof}

We are now in a position to prove Theorem~\ref{thm:asym}.
Its proof is very similar to that of Theorem~\ref{main}, where we replace Lemma~\ref{lem:almost} by Lemma~\ref{lem:fractional}, and Lemma~\ref{Lulemma2.3} is no longer needed.

\begin{proof}[Proof of Theorem~\ref{thm:asym}]
Let $\gamma>0$ and integers $1\le d\le k-1$ be given.
Note that $c^*_{k,d}\ge 1-(1-1/k)^{k-d}$ (see e.g.~\cite{AFHRRS}), and recall that Theorem~\ref{thm:asym} has been proved by Han and Treglown when $c^*_{k,d}\ge1/3$~\cite[Theorem 7.2]{HT}.
For the case $1\le d\le k/2$, we have $c^*_{k,d}\ge1/3$.
Indeed, since $(1-1/k)^k<1/e$, we have
$$c^*_{k,d}\ge1-(1-1/k)^{k-d}> 1-(1/e)^{1-d/k}\ge 1-(1/e)^{1/2}\ge 1/3.$$
These remarks imply that Theorem~\ref{thm:asym} holds for the case $1\le d\le k/2$.
Furthermore, we may assume $d\le k-2$ since the case $d=k-1$ has been proved by Han~\cite[Theorem 1.1]{Han15_mat}.
Next we deal with the case $k/2\le d\le k-2$.

Suppose $0<\sigma\ll \eta\ll\gamma_1,\gamma_2\ll\gamma\ll1/k$.
Let $H=(V,E)$ be an $n$-vertex $k$-graph with $\delta_d(H)\ge (c^*_{k,d}+\gamma)\binom{n-d}{k-d}$.
Recall that $c^*_{k,d}\ge 1-(1-1/k)^{k-d}>1-(1/e)^{1-d/k}$.
Thus, $c^*_{k,d}>1/4$ when $k/2\le d<\lceil2k/3\rceil$.
Consequently we have $\delta_d(H)\ge (1/4+\gamma_1)\binom{n-d}{k-d}$ when $k/2\le d<\lceil2k/3\rceil$ by the choice of $\gamma_1$, and $\delta_d(H)\ge \gamma_2 n^{k-d}$ when $\lceil 2k/3 \rceil \leq d \leq k-2$ by the choice of $\gamma_2$.

The proof is almost identical to the proof of Theorem~\ref{main} except that we use Lemma~\ref{lem:fractional} instead of Lemma~\ref{lem:almost}.
First we find an absorbing set in $H$.
Note that $2k-d-1\ge k+1$ since $d\le k-2$.
By Lemma~\ref{lem:abs1} and Lemma~\ref{lem:abs2}, there exists a matching $M$ in $H$ of size $|M|\le \eta n/k$ such that, for every subset $R\subseteq V\setminus V(M)$ with $|R|\le \sigma n$, the induced $k$-graph $H[R\cup V(M)]$ contains a matching covering all but at most $\max \{k+1, 2k-d-1\}=2k-d-1$ vertices.

Let $H_1 = H[V\setminus V(M)]$ and $n_1=n-k|M|\ge (1-\eta) n$. Next we find an almost perfect matching in $H_1$.
Recall that $\delta_d(H)\ge (c^*_{k,d}+\gamma)\binom{n-d}{k-d}$, then we have that
\[
\delta_d(H_1)\ge\delta_d(H)-(\eta n)n^{k-d-1}\ge(c^*_{k,d}+\gamma/2)\binom{n_1-d}{k-d},
\]
since $n$ is sufficiently large and $\eta\ll \gamma$.
Applying Lemma~\ref{lem:fractional} on $H_1$ with input $\sigma$ and $\gamma/2$ in place of $\gamma$, $H_1$ has a matching $M_1$ covering all but at most $\sigma n$ vertices of $V(H_1)$.
Denote by $U$ the set of these remaining vertices of $V(H_1)$.
Then $|U|\le \sigma n$.
Recall that $M$ is an absorbing matching in $H$, which implies that $H[V(M)\cup U]$ contains a matching $M_2$ covering all but at most $2k-d-1$ vertices.
Thus, $M_1\cup M_2$ is the desired matching.
\end{proof}

\section{Proof of Lemma~\ref{lem:abs1}}\label{sec3}
In this section we prove Lemma~\ref{lem:abs1}.
Because of the existence of the divisibility barrier (Construction~\ref{const1}), the absorbing lemma for perfect matching requires a minimum $d$-degree $\delta_d(H)\ge (1/2+o(1))\binom{n-d}{k-d}$.
However, we only have the minimum $d$-degree $\delta_d(H)\ge (1/4+o(1))\binom{n-d}{k-d}$.
Inspired by the divisibility barrier, we consider vertex partitions of $H$ and analyze them via the distribution of edges (robust edge-lattice).
Such approach was first used by Keevash and Mycroft~\cite{KM14}.

\subsection{Notation and preliminaries}
In order to state our results, we first introduce some notation.
Let $H$ be a $k$-graph on $n$ vertices.
We say that two vertices $u$ and $v$ are \emph{$(\beta, i)$-reachable} in $H$ if there are at least $\beta n^{ik-1}$ $(ik-1)$-sets $S \subseteq V(H)$ such that both $H[S \cup \{u\}]$ and $H[S \cup \{v\}]$ have perfect matchings.
We say that a vertex set $U$ is \emph{$(\beta, i)$-closed} in $H$ if every two vertices in $U$ are $(\beta, i)$-reachable in $H$.

We will work on a vertex partition $\cP=\{V_0, V_1, \dots, V_r\}$ of $V(H)$ for some integer $r\ge 1$.
We consider the $r$-dimensional vectors on the parts of $\cP$ \emph{except} $V_0$.
Formally, the \emph{index vector} $\bfi_{\cP}(S)\in \mathbb{Z}^r$ of a subset $S\subseteq V(H)$ with respect to $\cP$ is the vector whose coordinates are the sizes of the intersections of $S$ with each part of $\cP$ except $V_0$, namely, $\bfi_{\cP}(S)_{V_i}=|S\cap V_i|$ for $i\in [r]$.
We call a vector $\bfi\in \mathbb{Z}^r$ an $s$-vector if all its coordinates are nonnegative and their sum equals $s$, and we use $I_s^r$ to denote the set of all $s$-vectors.
Given $\mu >0$, a $k$-vector $\bfv$ is called a \emph{$\mu$-robust edge-vector} if there are at least $\mu n^k$ edges $e \in E(H)$ satisfying $\bfi_{\cP}(e)=\bfv$.
Let $I_{\cP}^{\mu}(H) \subseteq I_k^r$ be the set of all $\mu$-robust edge-vectors and let $L_{\cP}^{\mu}(H)$ be the \emph{lattice} (additive subgroup) generated by the vectors in $I_{\cP}^{\mu}(H)$.
For $i \in [r]$, let $\bfu_{i} \in \mathbb{Z}^r$ be the $i$-th \emph{unit vector}, namely, $\bfu_{i}$ has $1$ on the $i$-th coordinate and $0$ on other coordinates.
A \emph{transferral} is the vector $\bfu_{i}-\bfu_{j}$ for some $i \neq j$.

Suppose $I$ is a set of $k$-vectors in $\mathbb{Z}^r$ and $J$ is a set of vectors in $\mathbb{Z}^r$ such that any $\bfi\in J$ can be written as a linear combination of vectors in $I$, namely, $\bfi=\sum_{\bfv\in I}a_{\bfv}\bfv$.
We denote $C(r, k, I, J)$ as the maximum of $|a_{\bfv}|$, $\bfv\in I$, over all $\bfi\in J$ and $C(k,J) :=\max_{r\le k,\,I\subseteq I^r_k}C(r, k, I, J)$.

For a $k$-graph $H$, we first establish a partition $P$ of $V(H)$ and then study the so-called robust edge-lattice with respect to this partition.
The main tool for establishing the partition of $V(H)$ is the following lemma from~\cite[Lemma 3.3]{Han15_mat}.

\begin{lemma}[{\cite[Lemma 3.3]{Han15_mat}}]\label{partition}
Suppose $0<1/n\ll\beta \ll \varepsilon \ll \delta$.
Let $H$ be an $n$-vertex $k$-graph such that $\delta_{1}(H) \ge (\delta+k^{2}\varepsilon)\binom{n-1}{k-1}$.
Then there is a partition $\cP$ of $V(H)$ into $V_{0}, V_{1},\dots, V_{r}$ with $r \leq \lfloor 1/\delta \rfloor$ such that $|V_{0}|\leq \sqrt{\varepsilon}n$ and for any $i \in [r], |V_{i}|\geq \varepsilon^{2}n$ and $V_{i}$ is $(\beta, 2^{\lfloor 1/\delta \rfloor-1})$-closed in $H$.
\end{lemma}

The next result from~\cite[Lemma 3.4]{Han15_mat} shows that $V_i \cup V_j$ is closed in $H$ if $\bfu_{i}-\bfu_{j} \in L_{\cP}^{\mu}(H)$, which means that we can merge $V_i$ and $V_j$ and keep the closedness.

\begin{lemma}[{\cite[Lemma 3.4]{Han15_mat}}]\label{transferral}
Let $0 <\mu, \beta \ll \varepsilon \ll 1/i_{0}$, then there exist $0 < \beta' \ll \mu, \beta$ and an integer $t \geq i_{0}$ such that the following holds for sufficiently large $n$.
Suppose $H$ is an $n$-vertex $k$-graph, and $\cP = \{V_0,V_1, \dots, V_r\}$ is a partition with $r\leq i_0$ such that $|V_{0}|\leq \sqrt{\varepsilon}n$ and for any $i \in [r], |V_{i}|\geq \varepsilon^{2}n$ and $V_{i}$ is $(\beta, i_{0})$-closed in $H$.
If $\bfu_{i}-\bfu_{j} \in L_{\cP}^{\mu}(H)$, then $V_{i}\cup V_{j}$ is $(\beta', t)$-closed in $H$.
\end{lemma}

The following absorbing lemma from~\cite[Lemma 3.5]{Han15_mat} will be applied to prove Lemma~\ref{lem:abs1}.
Let $H$ be a $k$-graph and let $i\in k\mathbb{N}$.
For a $k$-vertex set $S$, we say a vertex set $T$ is an \emph{absorbing $i$-set for $S$} if $|T|=i$ and both $H[T]$ and $H[T\cup S]$ contain perfect matchings.

\begin{lemma}[{\cite[Lemma 3.5]{Han15_mat}}]\label{absorbing1}
Suppose $r\leq k$ and
\[
1/n \ll \alpha \ll \mu, \beta\ll 1/k, 1/t.
\]
Suppose that $\cP_{0} = \{V_0,V_1, \dots, V_r \}$ is a partition of $V(H)$ such that for $i \in [r]$, $V_{i}$ is $(\beta, t)$-closed.
Then there is a family $\mathcal{F}_{abs}$ of disjoint $tk^{2}$-sets with size at most $\beta n$ such that $H[V(\mathcal{F}_{abs})]$ contains a perfect matching and every $k$-vertex set $S$ with $\bfi_{\cP_{0}}(S)\in I_{\cP_{0}}^{\mu}(H)$ has at least $\alpha n$ absorbing $tk^{2}$-sets in $\mathcal{F}_{abs}$.
\end{lemma}

Let us give the following simple and useful proposition for considering $d$-degree together with $d'$-degree for some $d\neq d'$.
\begin{proposition}\label{prop:abs}
Let $0\leq d'\le d\leq k-1$ and $H$ be a $k$-graph on $n$ vertices.
If $\delta_d(H)\geq c\binom{n-d}{k-d}$ for some $0\leq c\leq 1$, then $\delta_{d'}(H)\geq c\binom{n-d'}{k-d'}$.
\end{proposition}
This proposition is straightforward since $\delta_{d'}(H)\geq \binom{n-d'}{d-d'} \delta_d(H)/\binom{k-d'}{d-d'}$.

\subsection{Proof of Lemma~\ref{lem:abs1}}
We now finish the proof of Lemma~\ref{lem:abs1}, which we restate below for convenience.

\begin{lemma*}
For $\min\{3,k/2\}\le d< \lceil 2k/3\rceil$ and $\gamma>0$, suppose $0<1/n\ll\alpha\ll \gamma,1/k$.
Let $H$ be an $n$-vertex $k$-graph with $\delta_d(H)\ge (1/4+\gamma)\binom{n-d}{k-d}$.
Then there exists a matching $M$ in $H$ of size $|M|\le \gamma n/k$ such that for any subset $R\subseteq V(H)\setminus V(M)$ with $|R|\le \alpha^{2} n$, $H[R\cup V(M)]$ contains a matching covering all but at most $k+1$ vertices.
\end{lemma*}

One of the key steps in the proof of Lemma~\ref{lem:abs1} is the following proposition. We postpone its proof to the end of this section.

\begin{proposition}\label{prop:absorb}
Given $\min\{3, k/2\} \leq d< \lceil 2k/3 \rceil$, suppose $0< 1/n\ll\mu \ll \varepsilon \ll \gamma$.
Let $H$ be an $n$-vertex $k$-graph with $\delta_{d}(H) \ge(1/4+\gamma)\binom{n-d}{k-d}$, and let $\cP = \{V_0,V_1, \dots, V_r \}$ be a partition of $V(H)$ with $r \leq 3$ such that $|V_{0}|\leq \sqrt{\varepsilon}n$ and for each $i \in [r], |V_{i}|\geq \varepsilon^{2}n$, and $L_{\cP}^{\mu}(H)$ contains no transferral.
Then for every $U\subseteq V(H)\setminus V_{0}$ with $|U|=k+2$, there exist $i,j \in [r]$ such that $\bfi_{\cP}(U)-\bfu_{i}-\bfu_{j}\in L_{\cP}^{\mu}(H)$.
\end{proposition}

Now we are ready to prove Lemma~\ref{lem:abs1}.
Here is a brief outline of the proof.
We apply Lemma~\ref{partition} to $H$ and obtain a partition $\cP$ of $V(H)$ such that each part is closed and not too small.
We then merge two parts $V_i$ and $V_j$ if the transferral $\bfu_i - \bfu_j\in L_{\cP}^{\mu}(H)$.
So we obtain a transferral-free partition $\mathcal{P}_0 = \{V_0,V_1,\dots, V_{r'}\}$, where $r'\le 3$.
Lemma~\ref{absorbing1} implies the existence of a family $\mathcal{F}_{abs}$ of disjoint absorbing $tk^2$-sets, which can be used to absorb a small collection of $k$-sets each with index vector in $L_{\cP_0}^{\mu}(H)$.
The key point is that as long as there are at least $k+2$ vertices uncovered, Proposition~\ref{prop:absorb} implies that we can ``absorb'' $k$ vertices, and our (quantitative) choice of the family $\mathcal{F}_{abs}$ allows us to proceed the absorption to reduce the number of uncovered vertices in a greedy manner.
The absorption terminates when there are at most $k+1$ vertices left uncovered.

\begin{proof}[Proof of Lemma~\ref{lem:abs1}]
Fix $\gamma>0$ and let $C:=C(k,I_{2k}^r)$.
We define additional constants such that
\[
0<1/n\ll \alpha\ll \beta'\ll\beta,\mu\ll \varepsilon \ll \gamma,1/k,1/t,1/C.
\]
Let $H = (V,E)$ be a $k$-graph with $\delta_d(H)\ge (1/4+\gamma)\binom{n-d}{k-d}$.
Note that we have $\delta_1(H)\ge (1/4+\gamma)\binom{n-1}{k-1}$ by Proposition~\ref{prop:abs}.
We first apply Lemma~\ref{partition} on $H$ with $\delta=1/4+\gamma/2$.
Then we get a partition $\mathcal{P}=\{V_0,V'_1,\dots,V'_{r'}\}$ with $r' \le 3$ such that $|V_0|\le \sqrt{\varepsilon} n$ and for any $i\in[r']$, $|V'_i|\ge \varepsilon^2 n$ and $V'_i$ is $(\beta,4)$-closed in $H$.
If $\mathbf{u}_i-\mathbf{u}_j\in L_{\mathcal{P}}^{\mu}(H)$ for some $i, j\in[r']$, $i\neq j$, then we merge $V'_i$ and $V'_j$ to one part, and by Lemma~\ref{transferral}, $V'_i\cup V'_j$ is $(\beta'', t')$-closed for some $\beta''> 0$ and $t' \ge 4$.
We greedily merge the parts until there is no transferral in the $\mu$-robust edge-lattice.
Let $\mathcal{P}_0 = \{V_0,V_1,\dots, V_{r}\}$ be the resulting partition for some $1\le r\le 3$.
Note that we may apply Lemma~\ref{transferral} at most twice, and we see that $V_i$ is $(\beta', t)$-closed for each $i\in[r]$ by the choice of $\beta'$.
We apply Lemma~\ref{absorbing1} on $H$ and get a family $\mathcal{F}_{abs}$ of disjoint $tk^2$-sets, and we conclude that $|V(\mathcal{F}_{abs})|\le tk^2\beta' n$, $H[V(\mathcal{F}_{abs})]$ contains a perfect matching $M_1$, and every $k$-vertex set $S$ with $\mathbf{i}_{\mathcal{P}_0}(S)\in I^{\mu}_{\mathcal{P}_0}(H)$ has at least $\alpha n$ absorbing $tk^2$-set in $\mathcal{F}_{abs}$.

Let $V'=V\setminus V(\mathcal{F}_{abs})$ and $H'=H[V']$. Now we find a matching $M_2$ in $H'$ as follows.
For each $\mathbf{v} \in I^{\mu}_{\mathcal{P}_0}(H)$, we greedily pick a matching $M_v$ of size $C\alpha^2n$ such that $\mathbf{i}_{\mathcal{P}_0}(e) = \mathbf{v}$ for every $e\in M_v$.
Then let $M_2$ be the union of $M_v$ for all $\mathbf{v}\in I^{\mu}_{\mathcal{P}_0}(H)$, and we have $V_0\cap V(M_2)=\emptyset$.
It is possible to pick $M_2$ because there are at least $\mu n^k$ edges $e$ with $\mathbf{i}_{\mathcal{P}_0}(e) = \mathbf{v} \in I^{\mu}_{\mathcal{P}_0}(H)$.
To be more precise, since $|I^{\mu}_{\mathcal{P}_0}(H)|\le \binom{k+r-1}{r-1}\le \binom{k+2}{2}$ and $\alpha\ll \beta' \ll \mu \ll 1/t,1/C$, we obtain
\[
|V(M_2)\cup V(\mathcal{F}_{abs})| \le k|I^{\mu}_{\mathcal{P}_0}(H)|C\alpha^2 n+tk^2 \beta' n<\mu n,
\]
which yields that the number of edges intersecting these vertices is less than $\mu n^k$, as required.

Next we build a matching $M_3$ to cover all vertices in $V_0\setminus V(\mathcal{F}_{abs})$.
Note that $|M_3|\le |V_0|\le \sqrt{\varepsilon}n$.
Specifically, when we greedily match a vertex $v\in V_0\setminus V(\mathcal{F}_{abs})$, we need to avoid at most $k|M_3|+|V(M_2)\cup V(\mathcal{F}_{abs})|\le k\sqrt{\varepsilon}n+\mu n\le 2k\sqrt{\varepsilon}n$ vertices, and thus at most $2k\sqrt{\varepsilon}n^{k-1}$ $(k-1)$-sets.
Since $\delta_1(H)> \gamma\binom{n-1}{k-1}>2k\sqrt{\varepsilon}n^{k-1}$, we can always find a desired edge containing $v$ and put it to $M_3$ as needed.

Let $M=M_1\cup M_2\cup M_3$.
It is easy to see that $M$ is a matching because $M_1, M_2$ and $M_3$ are pairwise vertex disjoint.
Now we prove that $M$ is the desired matching satisfying the conclusion of Lemma~\ref{lem:abs1}.
Note that $|M|\leq|M_3|+|V(M_2)\cup V(\mathcal{F}_{abs})|/k\le 2\sqrt{\varepsilon}n\leq \gamma n/k$ since $\varepsilon\ll \gamma,1/k$.
Consider any subset $R\subseteq V\setminus V(M)$ with $|R|\le \alpha^2 n$.
Fix any set $U\subseteq R$ of $k+2$ vertices, there exist $i,j\in [r]$ such that $\mathbf{i}_{\mathcal{P}_0}(U)- \mathbf{u}_i-\mathbf{u}_j\in L^{\mu}_{\mathcal{P}_0}(H)$ by Proposition~\ref{prop:absorb}.
Note that this does not guarantee that we can delete one vertex $u$ from $U\cap V_i$ and delete another vertex $v$ from $U\cap V_j$ such that $\mathbf{i}_{\mathcal{P}_0}(U\setminus \{u,v\})\in L^{\mu}_{\mathcal{P}_0}(H)$, because it is possible that $U \cap V_i=\emptyset$ or $U \cap V_j=\emptyset$ for the $i,j$ returned by the proposition.
By $d\ge 2$ and the degree condition, there is a vector $\mathbf{v}\in I^{\mu}_{\mathcal{P}_0}(H)$ such that $\mathbf{v}_{V_i} \ge 1$ and $\mathbf{v}_{V_j} \ge 1$.
Notice that $M_2$ contains $C\alpha^2 n$ edges with index vector $\mathbf{v}$.
Fix one such edge $e\in E_{\mathbf{v}}$ and two vertices $v_1\in e\cap V_i$, $v_2\in e\cap V_j$.
We delete $e$ from $M_2$ and let $U' = U\cup (e\setminus \{v_1,v_2\})$.
Clearly, $\mathbf{i}_{\mathcal{P}_0}(U')\in L^{\mu}_{\mathcal{P}_0}(H)$ and $|U'| = 2k$.
Hence, by the definition of $L^{\mu}_{\mathcal{P}_0}(H)$, there exist nonnegative integers $b_{\mathbf{v}}, c_{\mathbf{v}}$ for all $\mathbf{v}\in I^{\mu}_{\mathcal{P}_0}(H)$ such that
\[
\mathbf{i}_{\mathcal{P}_0}(U')=\sum\limits_{\mathbf{v}\in I^{\mu}_{\mathcal{P}_0}(H)}b_{\mathbf{v}} \mathbf{v}-\sum\limits_{\mathbf{v}\in I^{\mu}_{\mathcal{P}_0}(H)}c_{\mathbf{v}} \mathbf{v},\]
which implies that
\[
\mathbf{i}_{\mathcal{P}_0}(U')+\sum\limits_{\mathbf{v}\in I^{\mu}_{\mathcal{P}_0}(H)}c_{\mathbf{v}} \mathbf{v}=\sum\limits_{\mathbf{v}\in I^{\mu}_{\mathcal{P}_0}(H)}b_{\mathbf{v}} \mathbf{v}.
\]
We have that $b_{\mathbf{v}}, c_{\mathbf{v}}\le C$ from the definition of $C$.
For each $\mathbf{v}\in I^{\mu}_{\mathcal{P}_0}(H)$, we pick $c_{\mathbf{v}}$ edges in $M_2$ with index vector $\mathbf{v}$.
By the equation above, the union of these edges and $U'$ can be partitioned as a collection of $k$-sets, which contains exactly $b_{\mathbf{v}}$ $k$-sets $F$ with $\mathbf{i}_{\mathcal{P}_0}(F) = \mathbf{v}$ for each $\mathbf{v}\in I^{\mu}_{\mathcal{P}_0}(H)$.
We repeat the process at most $\alpha^2 n/k$ times until there are at most $k+1$ vertices left.
Note that for each $\mathbf{v}\in I^{\mu}_{\mathcal{P}_0}(H)$, our algorithm consumes at most $(1 + C)\alpha^2n/k < C\alpha^2n$ edges from $M_2$ with index vector $\mathbf{v}$, which is possible by the definition of $M_2$.
Furthermore, after the process, we obtain at most $\left(2+C|I^{\mu}_{\mathcal{P}_0}(H)|\right)\alpha^2n/k\le \left(2+C\binom{k+2}{2}\right)\alpha^2 n/k<\alpha n$ $k$-sets $S$ with $\mathbf{i}_{\mathcal{P}_0}(S)\in I^{\mu}_{\mathcal{P}_0}(H)$ since $\alpha\ll 1/k, 1/C$.
By the absorbing property of $\mathcal{F}_{abs}$, we can greedily absorb them by $\mathcal{F}_{abs}$ and get a matching $M_4$.
Thus, $H[R\cup V(M)]$ contains a matching covering all but at most $k + 1$ vertices.
\end{proof}

\subsection{The transferral-free lattices}
In this subsection we prove Proposition~\ref{prop:absorb}.
We study the lattice structure $L_{\cP}^{\mu}(H)$ when it contains no transferral.

Fix $1\leq p\leq k-1$ and any $p$-vector $\bfv$, we define its \emph{neighborhood} to be $N(\bfv):=\{\bfv'\colon\, \bfv+\bfv' \in L_{\cP}^{\mu}(H)\}$.
Note that the vectors in $N(\bfv)$ may contain negative coordinates.
Moreover, assuming $r=2$, we claim that $N(\bfv)\cap I_{k-p}^{2}\neq \emptyset$ for any $1\leq p\leq d$ and any $p$-vector $\bfv=(i, p-i)$.
Indeed, otherwise, let $\bfv$ be a $p$-vector such that $N(\bfv)\cap I^2_{k-p}=\emptyset$.
This implies that the number of edges in $H[V\setminus V_0]$ with index vector $\bfi$ such that $\bfi-\bfv\in I^2_{k-p}$ is at most $|I^2_{k-p}|\mu n^k\le 2^{k-p}\mu n^k$.
Let $A_{\bfv}$ be the set of all $p$-sets $S$ with $\bfi_{\cP}(S) =\bfv$, and thus $|A_{\bfv}| =\binom{|V_1|}{i}\binom{|V_2|}{p-i}\ge \binom{\varepsilon^2 n}{p}$.
By averaging, there is a $p$-set $S\in A_{\bfv}$ such that
\[
\deg_H(S)\le |I^2_{k-p}|\mu n^k/|A_{\bfv}|+|V_0|n^{k-p-1}\le 2^{k-p}\mu n^k/\binom{\varepsilon^2 n}{p}+\sqrt{\varepsilon}n^{k-p}<\gamma \binom{n-p}{k-p}
\]
by $\mu\ll \varepsilon\ll \gamma$.
Since $p\le d$ and by Proposition~\ref{prop:abs}, this contradicts that $\delta_d(H)\geq (1/4+\gamma)\binom{n-d}{k-d}$.
Note that a similar argument works for $r = 3$; namely, for any $p$-vector $\bfv = (i,i',p-i-i')$ with $1 \le p\le d$, $N(\bfv)\cap I^3_{k-p}\neq\emptyset$.

\begin{claim}\label{clm:r=2}
Given $\min\{3, k/2\} \leq d < \lceil 2k/3 \rceil$, suppose $0< \mu \ll \varepsilon \ll \gamma$.
Let $H$ and $\cP$ be as defined in Proposition~\ref{prop:absorb}.
If $r=2$, then $(2, -2)\in L_{\cP}^{\mu}(H)$ or $(3, -3)\in L_{\cP}^{\mu}(H)$.
If $r=3$, then $(-2, 1,1),(1,-2,1),(1,1,-2) \in L_{\cP}^{\mu}(H)$.
\end{claim}
\begin{proof}
Our proof is adapted from the proof of \cite[Claim 3.7]{Han15_mat}.
First assume that $r=2$.
Fix $(a_{0}, b_{0})\in I_{\cP}^{\mu}(H)$.
For the sake of a contradiction, assume that $(2, -2), (3, -3)\notin L_{\cP}^{\mu}(H)$.
Let $L_{0}$ be the sublattice (subgroup) of $L_{\cP}^{\mu}(H)$ such that $(a, b) \in L_{0}$ if $a+b=0$, and let $L_{k-d}=\{(a, b)\colon\,a+b=k-d\}$.
Let $t$ be the smallest positive integer such that $(t, -t)\in L_{0}$, then it is easy to see that $L_{0}$ is generated by $(t,-t)$.
By our assumption, $(1,-1),(2,-2), (3, -3)\notin L_{\cP}^{\mu}(H)$, and thus $t\geq 4$.
Let $t_{0}=\min\{t,k-d+1\}$.
It is easy to see that $L_{0}$ partitions $L_{k-d}$ into $t_0$ cosets $C_{0},C_1, \dots, C_{t_{0}-1}$ such that $C_{i}=(k-d-i, i)+L_{0}$\footnote{As usual, for a subgroup $H$ of a group $G$ and an element $x$ of $G$, $x+H$ denotes a coset of $H$.} for all $0\leq i\leq t_{0}-1$.
For any $0\leq j\leq d$ and $\bfv_{j}:=(d-j, j)$, by $a_0+b_0=k$, we have
\[
N(\bfv_{j})=(a_{0}, b_{0})-(d-j, j)+L_{0}=(k-d-(b_{0}-j), b_{0}-j)+L_{0}.
\]
This means that $N(\bfv_{j})=C_{i_{j}}$, where $i_{j}\equiv b_{0}-j \bmod t_{0}$.
We consider the following two cases depending on the value of $d$.

\vspace{4pt}
\emph{Case~$1$.} $d\ge 3$.
Since $(1,-1),(2,-2), (3, -3)\notin L_{\cP}^{\mu}(H)$, we claim that $N(\bfv_{0}), \dots, N(\bfv_{3})$ are pairwise disjoint.
Indeed, if $N(\bfv_0)\cap N(\bfv_3)\neq \emptyset$, say $\bfi\in N(\bfv_0)\cap N(\bfv_3)$, then we have $\bfi+\bfv_0, \bfi+\bfv_3\in L_{\cP}^{\mu}(H)$ and thus $(3,-3)=\bfv_0-\bfv_3\in L_{\cP}^{\mu}(H)$, a contradiction.
Other cases can be dealt with similarly.
Note that for any $0\leq j\leq d$ and $\bfv_{j}:=(d-j, j)$, we have $N(\bfv_{j})=C_{i_{j}}$, where $i_{j}\equiv b_{0}-j \bmod t_{0}$.
So we have $t_0 \geq 4$.
For $j=0, \dots, 3$, consider the following sums
\[
\sum_{(k-d-i_j,i_j)\in C_{i_j},\, 0\leq i_{j}\leq k-d} \binom{|V_{1}|}{k-d-i_{j}} \binom{|V_{2}|}{i_{j}},
\]
and note that their sum is at most $\binom{n-|V_0|}{k-d}$.
By the pigeonhole principle, there exists $j'$ such that the $j'$-th sum is at most $\frac{1}{4} \binom{n-|V_{0}|}{k-d}$.
This implies that
\[
\delta_{d}(H)\leq \frac{1}{4} \binom{n-|V_{0}|}{k-d}+|I_{k-d}^{2}|\mu n^{k}/ \binom{\varepsilon^{2}n}{d}+|V_{0}|n^{k-d-1}<(1/4+\gamma)\binom{n-d}{k-d}
\]
by $\mu \ll\varepsilon \ll \gamma$, a contradiction.

\vspace{4pt}
\emph{Case~$2$.} $d=2$. In this case we have $k=4$ by $\min\{3, k/2\} \leq d < \lceil 2k/3 \rceil$.
Then $k-d=2$, $t_0=3$, and for $i=0,1,2$, $C_i\cap I_2^2=\{(2-i,i)\}$.
Since $(1,-1),(2,-2)\notin L_{\cP}^{\mu}(H)$, we know that for $j=0,1,2$, $N(\bfv_j)\cap I_2^2=C_{i_j}\cap I_2^2=\{(2-i_j,i_j)\}$ are pairwise distinct.
So $\{N(\bfv_j)\cap I_2^2\}_{j=0,1,2}=\{\{(2,0)\}, \{(1,1)\}, \{(0,2)\}\}$.
Note that
\begin{align}
\min\left\{\binom{|V_1|}{2}, |V_1||V_2|, \binom{|V_2|}{2}\right\}&\leq \max\limits_{x\in (0,1)}\min\left\{\binom{xn}{2},xn(1-x)n,\binom{(1-x)n}{2}\right\}\nonumber \\
 &\leq \max\limits_{x\in (0,1)}\left(\min\left\{x^2,2x(1-x),(1-x)^2\right\}+\gamma/4\right)\binom{n-2}{2} \nonumber \\
 &=(1/4+\gamma/4)\binom{n-2}{2} \nonumber
\end{align}
since $n$ is large enough.
By averaging, we get that
\[
\delta_2(H)\le (1/4+\gamma/4)\binom{n-2}{2}+2^2\mu n^4/\binom{\varepsilon^2 n}{2}+|V_0|n<(1/4+\gamma)\binom{n-2}{2}
\]
by $\mu\ll \varepsilon\ll \gamma$, a contradiction.

Second assume that $r=3$.
Indeed, in this case, by $\min\{3,k/2\}\le d<\lceil 2k/3\rceil$ and Proposition \ref{prop:abs}, we have $\delta_{2}(H)\geq (1/4+\gamma)\binom{n-2}{k-2}$.
Consider the set of 2-vectors
\[
I^3_2=\{(2,0,0), (0,2,0), (0,0,2), (1,1,0), (1,0,1),(0,1,1)\}.
\]
Note that $N((1,1,0))$, $N((1,0,1))$, and $N((0,1,1))$ are pairwise disjoint because $L_{\cP}^{\mu}(H)$ contains no transferral.
Similarly, $N((2,0,0)) \cap N((1,1,0))=\emptyset$ and $N((2,0,0)) \cap N((1,0,1))=\emptyset$ (and similar equations hold for other vectors).
Moreover, recall that $N(\bfv) \cap I^3_{k-2}\neq \emptyset$ for all $\bfv \in I^3_2$.
Thus, $I^3_{k-2}$ are partitioned into classes $C_{1}', \cdots, C_{m}'$ for $m \geq 3$ where each class has the form $N(\bfv) \cap I^3_{k-2}$ for some (not necessarily unique) $\bfv \in I^3_2$.
If $m \geq 4$, then consider the sums $\sum_{(j_1,j_2,j_3) \in C_i'} \binom{|V_{1}|}{j_1} \binom{|V_{2}|}{j_2}\binom{|V_{3}|}{j_3}$ for $i=1,\dots,m$ and note that their sum is at most $\binom{n-|V_{0}|}{k-2}$.
Similar to the previous cases, by averaging, we have
\[
\delta_{2}(H)\leq \frac{1}{m} \binom{n-|V_{0}|}{k-2}+|I_{k-2}^{3}|\mu n^{k}/ \binom{\epsilon^{2}n}{2}+|V_{0}|n^{k-3}<(1/4+\gamma)\binom{n-2}{k-2}
\]
by $\mu \ll\epsilon \ll \gamma$, a contradiction.
Otherwise, $m=3$.
Since $N((1,1,0))$, $N((1,0,1))$ and $N((0,1,1))$ must be in different classes, we know that
\[
N((1,1,0))=N((0,0,2)), N((1,0,1))=N((0,2,0)), \mbox{~and~} N((0,1,1))=N((2,0,0)),
\]
which implies that $(-2, 1,1),(1,-2,1),(1,1,-2) \in L_{\cP}^{\mu}(H)$.
\end{proof}

\begin{proof}[Proof of Proposition~\ref{prop:absorb}]
Given such a $k$-graph $H$ and a partition $\cP$, the conclusion is trivial for $r=1$. So we may assume that $r\in\{2,3\}$.
Applying Claim \ref{clm:r=2}, we conclude that $(2, -2)$ or $(3, -3)\in L_{\cP}^{\mu}(H)$ (for $r=2$), or $(-2,1,1),(1,-2,1),(1,1,-2) \in L_{\cP}^{\mu}(H)$ (for $r=3$).

If $r=2$, then fix any $U\subseteq V(H)\setminus V_{0}$ with $\bfi_{\cP}(U)=(a, k+2-a)$ for some $0 \leq a \leq k+2$ and pick any $(a_{0}, b_{0})\in I_{\cP}^{\mu}(H)$.
We distinguish two cases.
First we assume that $(2,-2)\in L_{\cP}^{\mu}(H)$, then we have $(a_0+2i,b_0-2i)\in L_{\cP}^{\mu}(H)$ for any integer $i$.
Our goal is to show that there exist $i,j\in [2]$ such that $\bfi_{\cP}(U)-\bfu_i-\bfu_j\in L_{\cP}^{\mu}(H)$.
It suffices to prove that $(a-1,k+1-a)\in L_{\cP}^{\mu}(H)$ or $(a,k-a)\in L_{\cP}^{\mu}(H)$.
Note that $a -1$ and $a$ have different parities, so exactly one of $(a-1,k+1-a)$ and $(a,k-a)$ is in $L_{\cP}^{\mu}(H)$.
Second we assume that $(3,-3)\in L_{\cP}^{\mu}(H)$, then we have $(a_0+3i,b_0-3i)\in L_{\cP}^{\mu}(H)$ for any integer $i$.
Our goal is to show that there exist $i,j\in [2]$ such that $\bfi_{\cP}(U)-\bfu_i-\bfu_j\in L_{\cP}^{\mu}(H)$.
It suffices to prove that $(a-2,k+2-a)$, $(a-1,k+1-a)$, or $(a,k-a)$ is in $L_{\cP}^{\mu}(H)$.
Note that exactly one of the three consecutive integers $a-2$, $a-1$, and $a$ is congruent with $a_0$ modulo $3$.
Thus, exactly one of $(a-2,k+2-a)$, $(a-1,k+1-a)$, and $(a,k-a)$ is in $L_{\cP}^{\mu}(H)$.

Next we assume $r=3$ and fix any $U\subseteq V(H)\setminus V_{0}$ with $\bfi_{\cP}(U)=(x_1,x_2,x_3)$ for some nonnegative integers $x_1+x_2+x_3=k+2$.
Pick any $(y_1,y_2,y_3) \in I_{\cP}^{\mu}(H)$ and let $z_j= x_j-y_j$ for $j \in [3]$.
Note that exactly one of the three consecutive integers $z_2-z_1$, $z_2-z_1+1$, and $z_2-z_1+2$ is divisible by $3$.
Let $i,j \in \{1,3\}$ such that $\bfv:=(z_1',z_2',z_3')=(z_1,z_2,z_3)-\bfu_i-\bfu_j$ satisfies that $z_2'-z_1'$ is divisible by $3$.
Let $m'= (z_2'-z_1')/3$ and $m=m'-z_2'$.
Note that $z_1'+z_2'+z_3'=0$; then it is easy to see that
\[
\bfi_{\cP}(U)-\bfu_i-\bfu_j-(y_1,y_2,y_3)=\bfv=m'(-2,1,1)-m(1,1,-2) \in L_{\cP}^{\mu}(H).
\]
Thus, $\bfi_{\cP}(U)-\bfu_{i}-\bfu_{j}\in L_{\cP}^{\mu}(H)$, and the proof is complete.
\end{proof}

\section{Proof of Lemma~\ref{lem:abs2}}\label{sec4}
We start with the following definition.
Given a set $S$ of $2k-d$ vertices, an edge $e \in E(H)$ that is disjoint from $S$ is called \emph{$S$-absorbing} if there are two disjoint edges $e_1$ and $e_2$ in $E(H)$ such that $|e_1 \cap S|=k-\lfloor d/2\rfloor$, $|e_1 \cap e|=\lfloor d/2\rfloor$, $|e_2 \cap S| =k-\lceil d/2\rceil$, and $|e_2 \cap e|=\lceil d/2\rceil$ (see Figure~\ref{fig:abs}).
Note that this is not the absorbing structure in the usual sense because $e_1 \cup e_2$ misses $k-d$ vertices of $S \cup e$.
Let us explain how such absorbing structure works.
Consider a matching $M$ and a $(2k-d)$-set $S$, $V(M) \cap S=\emptyset$.
If $M$ contains an $S$-absorbing edge $e$, then one can ``absorb'' $S$ into $M$ by swapping $e$ for $e_1$ and $e_2$ ($k-d$ vertices of $e$ become uncovered).

\begin{figure}[!ht]
\begin{center}
\begin{tikzpicture}
[inner sep=2pt,
   vertex/.style={circle, draw=black, fill=white},
   ]
\begin{pgfonlayer}{background}    
\node[trapezium, draw,rounded corners,inner xsep=1pt, trapezium left angle=90, trapezium right angle=110,minimum height = 2.5cm,fill=blue!8] (labelTransform) at (-1.7,0.5) {};
\node[trapezium, draw,rounded corners,inner xsep=0.5pt, trapezium left angle=70, trapezium right angle=120,minimum height = 2.5cm,fill=blue!8] (labelTransform) at (0.4,0.5) {};

\draw[rounded corners,dashed] (-2.5,1) rectangle (2.5, 2.5);  
\draw[rounded corners] (-2.5,0) rectangle (2.3, -1.3);  
\end{pgfonlayer}
\node at (3, 1.7) {$S$};
\node at (3, -0.7) {$e$};
\node at (-1.4, 0.5)[color=blue] {$e_1$};
\node at (0.5, 0.5)[color=blue] {$e_2$};
\node at (-1.4, 1.3) {\small{$k-\lfloor d/2\rfloor$}}; 
\node at (-1.7, -0.4) {\small{$\lfloor d/2\rfloor$}}; 
\node at (0.9, 1.3) {\small{$k-\lceil d/2\rceil$}}; 
\node at (0.1, -0.4) {\small{$\lceil d/2\rceil$}}; 
\node at (1.2,-0.4) [vertex, inner sep=1pt] {};
\node at (2,-0.4) [vertex, inner sep=1pt] {};
\node at (1.65,-0.4) {$\cdots$};
\draw[decorate,decoration={brace,raise=4pt}] (2,-0.4) -- (1.2,-0.4); 
\node at (1.6,-0.8) {\small{$k-d$}};
\end{tikzpicture}
\caption{\small An $S$-absorbing edge $e$.}
\label{fig:abs}
\end{center}
\end{figure}
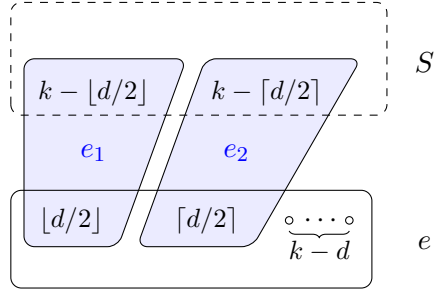

Below we restate and prove Lemma~\ref{lem:abs2}.
\begin{Lemma*}
For $\lceil 2k/3 \rceil \leq d \leq k-1$ and $\gamma'>0$, suppose $0<1/n\ll\beta\ll \gamma',1/k$.
Let $H$ be an $n$-vertex $k$-graph with $\delta _d (H) \geq \gamma' n^{k-d}$, then there exists a matching $M'$ in $H$ of size $|M'| \leq \beta n/k$ and such that for any subset $R\subseteq V(H)\setminus V(M')$ with $|R|\leq\beta^2 n$, $H[V(M')\cup R]$ contains a matching covering all but at most $2k-d-1$ vertices.
\end{Lemma*}

The proof of Lemma~\ref{lem:abs2} follows an idea of R\"{o}dl, Ruci\'{n}ski, and Szemer\'{e}di~\cite[Fact 2.2, Fact 2.3]{RRS09}, which was used for the case $d=k-1$.

\begin{proof}[Proof of Lemma~\ref{lem:abs2}]
Let $k,d,\gamma'$ be given and let $\beta$ be a constant such that $\beta\le\beta_0=(\gamma')^3/(12(k+1)!)$.
Denote $\ell_1:=\lfloor d/2\rfloor$ and $\ell_2:=\lceil d/2\rceil$, then $\ell_1+\ell_2=d$.
Let $H$ be an $n$-vertex $k$-graph with $\delta_d(H)\ge \gamma' n^{k-d}$.
Note that we have $\delta_{k-\ell_1}(H)\ge \gamma' n^{\ell_1}$ and $\delta_{k-\ell_2}(H)\ge \gamma' n^{\ell_2}$ from Proposition~\ref{prop:abs} and by $k-\ell_2\le k-\ell_1 \leq d$.

We first show that there
are many $S$-absorbing edges in $H$ for any $(2k-d)$-set of vertices $S$.
\begin{claim}
For every $S=\{u_1,\dots,u_{2k-d}\}\in\binom{V(H)}{2k-d}$, there are at least $\frac{1}{2}(\gamma')^3n^k/k!$ $S$-absorbing edges.
\end{claim}
\begin{proof}
Fix $k-\ell_1$ vertices $u_1,\dots,u_{k-\ell_1}$ in $S$ and let $e_1,e_2$ be as in the definition of $S$-absorbing edges.
Let us count only those $S$-absorbing edges $e$ for which the corresponding edge $e_1$ contains $u_1,\dots,u_{k-\ell_1}$.
We count the ordered $k$-tuples of distinct vertices $(v_1,\dots,v_k)$ such that $e = \{v_1,\dots, v_k\}$ is disjoint from $S$, $e_2\cap e=\{v_{\ell_1+1},\dots,v_{d}\}$, and $e_1=\{u_1,\dots,u_{k-\ell_1},v_1,\dots,v_{\ell_1}\}$, and divide the result by $k!$.

Note that $\{v_1,\dots,v_{\ell_1}\}$ must be a neighbor of an already fix $(k-\ell_1)$-tuple of vertices, so there are at least $\delta_{k-\ell_1}(H)-2kn^{\ell_1-1}$ choices for the $\ell_1$-tuple.
Recall that $\ell_2=d-\ell_1$.
Having selected $v_1,\dots, v_{\ell_1}$, $\{v_{\ell_1+1},\dots,v_{d}\}$ must be a neighbor of an already fixed $(k-\ell_2)$-tuple of vertices, so there are at least $\delta_{k-\ell_2}(H)-2kn^{\ell_2-1}$ choices for the $\ell_2$-tuple.
Having selected $v_1,\dots, v_d$, $\{v_{d+1},\dots,v_{k}\}$ must be a neighbor of an already fixed $d$-tuple of vertices, so there are at least $\delta_{d}(H)-2kn^{k-d-1}$ choices for the $(k-d)$-tuple.
Altogether since $n$ is large enough, there are at least $\left(\delta_{k-\ell_1}(H)-2kn^{\ell_1-1}\right)\left(\delta_{k-\ell_2}(H)-2kn^{\ell_2-1}\right)\left(\delta_{d}(H)-2kn^{k-d-1}\right)\ge \frac{1}{2}(\gamma')^3n^{\ell_1+\ell_2+k-d}=\frac{1}{2}(\gamma')^3n^k$ choices of the desired ordered $k$-tuples.
So there are at least $\frac{1}{2}(\gamma')^3n^k/k!$ $S$-absorbing edges in $H$.
\end{proof}

Now we pick the absorbing matching $M'$.
Select a random subset $M$ of $E(H)$, where each edge is chosen independently with probability $p=\beta n^{1-k}/(k+1)$.
Then, the expected size of $M$ is at most $\binom{n}{k}p < \beta n/(k+1)!$, and the expected number of intersecting pairs of edges in $M$ is at most $n^{2k-1}p^2<\beta^2 n$.
Hence, by Markov's inequality (see, e.g., \cite[inequality (1.3)]{JLR00}), with probability at least $1-1/2-1/k!$, $|M| \leq \beta n/k$ and $M$ contains at most $2\beta^2n$ intersecting pairs of edges.
Moreover, for every $(2k-d)$-set of vertices $S$, let $X_S$ be the number of $S$-absorbing edges in $M$.
Then we have
\[
\mathbb{E}(X_S) \geq p \cdot \frac{1}{2}(\gamma')^3n^k /k! =\frac{\beta (\gamma')^3 n}{2(k+1)!}.
\]
By Chernoff's bound (see, e.g., \cite[Theorem 2.1]{JLR00}), with probability $1-o(1)$, we have that $X_S \geq \frac{1}{2} \mathbb{E}(X_S) \geq \frac{\beta (\gamma')^3 n}{4(k+1)!}$ for all $(2k-d)$-sets $S$ in $H$.

Thus, there is an $M \subseteq E(H)$ satisfying all the properties above.
We delete one edge from each intersecting pairs of edges and denote the resulting subset by $M'$, which is a matching.
So $|M'| \leq \beta n/k$, and for every $(2k-d)$-set of vertices $S$, $M'$ contains at least $\frac{\beta (\gamma')^3 n}{4(k+1)!}-2 \beta^2 n \ge\beta^2 n$ $S$-absorbing edges by the definition of $\beta$.

It remains to show that, for any $R\subseteq V(H)\setminus V(M')$ with $|R|\leq\beta^2 n$, $H[V(M')\cup R]$ contains a matching covering all but at most $2k-d-1$ vertices.
Fix $R\subseteq V(H)\setminus V(M')$ with $|R|\leq\beta^2 n$ and any $(2k-d)$-tuple $S$ of $R$, then there are at least $\beta^2 n$ $S$-absorbing edges in $M'$.
Take an $S$-absorbing edge $e$, we replace $M'$ by $M'_S:=(M'\setminus \{e\}) \cup \{e_1, e_2\}$, decreasing the number of uncovered vertices of $R$ by $k$. Since we have at most $\beta^2 n/k$ iterations, there will always be an $S$-absorbing edge available in $M'$.
In the end, we have at most $2k-d-1$ vertices left uncovered in $H[V(M')\cup R]$ and we are done.
\end{proof}

\bibliographystyle{plain}

\end{document}